\documentclass[a4paper]{article}
\usepackage{amsmath,amsthm,amssymb}
\usepackage{graphicx}
\usepackage{multirow}
\usepackage{dcolumn}
\newtheorem{Def}{Definition}
\newtheorem{theorem}[Def]{Theorem}
\newtheorem*{maintheorema}{Main Theorem A}
\newtheorem*{maintheoremb}{Main Theorem B}
\newtheorem{lemma}[Def]{Lemma}
\newtheorem{prop}[Def]{Proposition}
\newtheorem*{corollary}{Corollary C}
\newtheorem{remark}[Def]{Remark}
\newtheorem{claim}[Def]{Claim}
\makeatletter
 
\renewcommand{\@cite}[2]{{#1\if@tempswa , #2\fi}}
\makeatother

\title{On roots of Dehn twists.}
\author{Naoyuki Monden}
\date{}
\begin{document}
\maketitle

\begin{abstract}
Margalit and Schleimer constructed nontrivial roots of the Dehn twist about a nonseparating curve.
We prove that the conjugacy classes of roots of the Dehn twist about a nonseparating curve correspond to the conjugacy classes of periodic maps with certain conditions.
Furthermore, we give data set which determine the conjugacy class of a root.
As a consequence, we can find the minimum degree and the maximum degree, and show that the degree must be odd.
Also, we give Dehn twist expression of the root of degree 3.
\end{abstract}

\section{Introduction}
Let $\Sigma_{g,b}$ denote a compact, oriented surface of genus $g$ with $b \ (=0,2)$ boundary components, and let ${\cal M}_{g,b}$ denote its mapping class group, the group of isotopy class of orientation-preserving homeomorphisms which are allowed to permute the points on the boundary.
The isotopies are also required to permute the points on the boundary.
If $b=0$, we omit it from the notation.
Let $t_{C}$ be the (left handed) Dehn twist about a simple closed curve $C$ on $\Sigma_{g+1}$ $(g+1\geq 2)$.

It is a natural question whether $t_{C}$ has a root.
In other words, does there exist $h\in {\cal M}_{g+1}$ such that for integer degree $n>1$ $t_{C}=h^{n}$?
If $C$ is a separating curve, it is easy to find roots.
For example, the half twist about $C$ is a root of degree $2$.
However, when $C$ is a nonseparating curve, roots were not discovered until recently.
In 2009, Margalit and Schleimer [MS] discovered roots of the Dehn twist $t_{C}$ about a nonseparating curve $C$ of degree $2g+1$.

Matsumoto and Montesinos [MM1] introduced a certain type of mapping class of $\Sigma_{g+1}$ which were called by them pseudo-periodic maps of  negative twist.
Furthermore, they gave the complete set of conjugacy invariant for a pseudo-periodic map of negative twist.
In Section 2 we introduce the definition and the complete set of conjugacy invariant of pseudo-periodic maps of  negative twist .
In Section 3 we show that inverses of roots of Dehn twists are pseudo-periodic maps of negative twist, and prove Main Theorem A by using the work of [MM1].
\begin{maintheorema}
Let $C$ be a nonseparating curve.
Let $\partial_{\nu}\Sigma_{g,2}$ $(\nu=0,1)$ be the boundary of $\Sigma_{g,2}$, which we write $\overrightarrow{\partial_{\nu}\Sigma_{g,2}}$ to emphasize its orientation.

The conjugacy classes $(\in {\cal M}_{g+1})$ of roots of $t_{C}$ of degree $n$ correspond to the conjugacy classes $(\in {\cal M}_{g,2})$ of periodic maps of order $n$ which satisfy following condition:

The action on $\overrightarrow{\partial_{\nu}\Sigma_{g,2}}$ $(\nu=0,1)$ is the rotation angle of $2\pi \delta^{\nu}/n$ such that $\delta^{0}+\delta^{1}=n-1$
$(Here \ \gcd(\delta^{\nu},n)=1, 1\leq \delta^{\nu}\leq n-2 \ (\nu=0,1))$.
\end{maintheorema}
Moreover, by combining the Riemann-Hurwitz formula, Nielsen's theorem [Ni] and Main Theorem A we give the following data set which determines the conjugacy class of a root.
\begin{maintheoremb}\label{valency}
The conjugacy class $(\in {\cal M}_{g+1})$ of a root $h$ of $t_{C}$ of degree $n$ is determined by the data set $[n,g^{\prime},(\sigma^{0},\sigma^{1});(\sigma_{1},\lambda_{1}),\ldots,(\sigma_{k},\lambda_{k})]$ which satisfies the following conditions:
\begin{description}
\item[(1)] $2g/n=2g^{\prime}+\sum_{j=1}^{k}(1-1/\lambda_{i})$,
\item[(2)] $\sum_{j=1}^{k}\sigma_{j}n/\lambda_{j}+\sigma^{0}+\sigma^{1}\equiv 0(\pmod{n})$,
\item[(3)] $\sigma^{0}+\sigma^{1}+\sigma^{0}\sigma^{1}\equiv 0(\pmod{n})$.
\end{description}
\end{maintheoremb}
As a corollary of Theorem~\ref{valency} we obtain the following result.
\begin{corollary}
the range of degree $n$ is $3\leq n\leq 2g+1$ and $n$ is odd.
\end{corollary}
Thus the roots constructed by Margalit and Schleimer have the maximum degree $2g+1$.

In Section 4 we give Dehn twist expression of the root of degree 3.
In Section 5 we consider a root of image of the Dehn twist about a nonseparating curve to ${\rm Sp}(2(g+1),\mathbb{Z})$.
We show that the degree must be odd.
In Section 6 we consider a root of the Dehn twist about a nonseparating curve in the mapping class group of a surface with boundary components and punctures.

Recently, McCullough and Rajeevsarathy [MR] has independently proved the result similar to Main Theorem A, B and Corollary C.
They gave the data set similar to Theorem B in Theorem 1.1 and Main Theorem A was contained in the proof of Theorem 1.1.
They proved these results without work of [MM1].
As a corollary of Theorem 1.1 they showed the same results as Corollary C.
Moreover, they showed the follow results:
In Corollary 3.1 they showed that $t_{C}$ has a root of degree $n$ for all $g+1>(n-2)(n-1)/2$.
In Corollary 3.2 they showed that when $n$ is prime and $g=(n-2)(n-1)/2$, $t_{c}$ does not have a root of degree $n$ \ (Corollary 3.2).
In Theorem 4.2 they proved that if $n\geq g\geq 4$ then the root is one of 2 types of roots \ (Theorem 4.2).

The author thanks Darryl McCullough and Kashyap Rajeevsarathy for their comments on the content of this paper.

\section{Preliminary}
Hereafter, all surfaces will be oriented, and all homeomorphisms between them will be orientation-preserving.
For us, a Dehn twist means a left-handed Dehn twist.
Let $\Sigma_{g+1}$ be a closed connected surface of genus $g+1\geq 2$.

\subsection{Pseudo-periodic map and screw number}
\begin{Def}
A homeomorphism $f:\Sigma_{g+1}\rightarrow \Sigma_{g+1}$ is a \underline{pseudo-periodic map} isotopic to a homeomorphism $f^{\prime}:\Sigma_{g+1}\rightarrow \Sigma_{g+1}$ which satisfies the following conditions:
\begin{description}
\item[(1)] There exists a disjoint union of simple closed curves ${\cal C}=C_{1}\bigcup C_{2}\bigcup \ldots \bigcup C_{r}$ on $\Sigma_{g+1}$ such that $f^{\prime}({\cal C})={\cal C}$ $({\cal C} \ might \ be \ empty)$.
\item[(2)] The restriction $f^{\prime}|_{(\Sigma_{g+1}-{\cal C})}:\Sigma_{g+1}-{\cal C}\rightarrow \Sigma_{g+1}-{\cal C}$ is isotopic to a periodic map.
\end{description}
\end{Def}

We call $\{C_{i}\}_{i=1}^{r}$ an \underline{admissible system of cut curves} if each connected component of $\Sigma_{g+1}-{\cal C}$ has negative Euler number.
If $\Sigma_{g+1}$ has negative Euler number, such a system always exists in the isotopy class of $f$.

Nielsen [Ni] introduced the "screw number" $s(C_{i})$ for each component $C_{i}$ of admissible system of cut curves ${\cal C}$. The definition is as follows:

We may assume $f({\cal C})={\cal C}$.
Let $C_{i}$ be a fixed cut curve in the system.

Let $\alpha_{i}$ be the smallest positive integer $\alpha_{i}$ such that $f^{\alpha_{i}}(\overrightarrow{C_{i}})=\overrightarrow{C_{i}}$ where $\overrightarrow{C_{i}}$ denotes $C_{i}$ with an orientation assigned, so $f^{\alpha_{i}}$ preserves the (arbitrarily fixed) direction of $C_{i}$.

Let $b_{i}$ and $b_{i}^{\prime}$ be the connected components of $\Sigma_{g+1}-{\cal C}$ such that $C_{i}$ belongs to their adherence in $\Sigma_{g+1}$.
Let $\beta_{i}$ (resp. $\beta_{i}^{\prime}$) be the smallest positive integer such that $f^{\beta_{i}}(b_{i})=b_{i}$ (resp. $f^{\beta_{i}^{\prime}}(b_{i}^{\prime})=b_{i}^{\prime}$).
Clearly $\alpha_{i}$ is a common multiple of $\beta_{i}$ and $\beta_{i}^{\prime}$.

Since $f|_{(\Sigma_{g+1}-{\cal C})}$ is isotopic to a periodic map, there exists a positive integer $n_{b_{i}}$ such that $(f^{\beta_{i}}|_{b_{i}})^{n_{b_{i}}}\cong id|_{b_{i}}$.
We choose the smallest number among such integers and denote it by $n_{b_{i}}$ again.
Likewise we choose $n_{b_{i}^{\prime}}(>0)$ for $b_{i}^{\prime}$.

Let $L_{i}$ be the least common multiple of $n_{b_{i}}\beta_{i}$ and $n_{b_{i}^{\prime}}\beta_{i}^{\prime}$.
Then $f^{L_{i}}$ is isotopic to the identity on $b_{i}$ and $b_{i}^{\prime}$.
Thus, on the union $b_{i}\bigcup C_{i}\bigcup b_{i}^{\prime}$, $f^{L_{i}}$ is isotopic to the result of a number full Dehn twists performed about $C_{i}$.
Let $e_{i}$ ($\in \mathbb{Z}$) be this number of full twists.

\begin{Def}
The rational number $e_{i}\alpha_{i}/L_{i}$ is called the \underline{screw number} of $f$ about $C_{i}$ and is denoted by $s(C_{i})$.
\end{Def}
\begin{Def}
A pseudo-periodic map $f$ is said to be of \underline{negative twist} if $s(C_{i})<0$ for each component $C_{i}$ of admissible system of cut curves ${\cal C}$.
\end{Def}
\begin{Def}
The curve $C_{i}$ is said to be \underline{amphidrome} if $\alpha_{i}$ is even and $f^{\alpha_{i}/2}(\overrightarrow{C_{i}})=-\overrightarrow{C_{i}}$, and non-amphidrome otherwise.
\end{Def}
\begin{theorem}[{\rm [MM1]}, Theorem 2.2]\label{MM}
The conjugacy class of a negative twist $f:\Sigma_{g+1}\rightarrow \Sigma_{g+1}$ is determined by following data:
\begin{itemize}
\item An admissible system of cut curves ${\cal C}=\displaystyle \bigcup_{i}^{r} C_{i}$ on $\Sigma_{g+1}$.
\end{itemize}
For all $i=1,\ldots,r$,
\begin{itemize}
\item $\alpha_{i}$ : the smallest positive integer $\alpha_{i}$ such that $f^{\alpha_{i}}(\overrightarrow{C_{i}})=\overrightarrow{C_{i}}$,
\item the screw number $s(C_{i})$ of $f$, and
\item $C_{i}$'s character of being amphidrome or not.
\end{itemize}
For each connected component $b$ of $\Sigma_{g+1}-{\cal C}$,
\begin{itemize}
\item $\beta$ : the smallest positive integer such that $f^{\beta}(b)=b$,
\item $n_{b}$ : the smallest positive integer such that $(f^{\beta}|_{b})^{n_{b}}\cong id|_{b}$, and
\item the conjugacy class of a periodic map $f^{\beta}|_{b}$.
\item The action of $f$ on the oriented graph $G_{{\cal C}}$ whose vertices and edges correspond to connected components of $\Sigma_{g+1}-{\cal C}$ and $\{C_{i}\}_{i=1}^{r}$.
\end{itemize}
\end{theorem}

\subsection{Valency}

Let $\Sigma$ be an oriented connected surface and let $f:\Sigma \rightarrow \Sigma$ be a periodic map of order $n>1$.
Let $p$ be a point on $\Sigma$.
There is a positive integer $\alpha(p)$ such that the points $p,f(p),\ldots,f^{\alpha(p)-1}(p)$ are mutually distinct and $f^{\alpha(p)}(p)=p$.
If $\alpha(p)=n$, we call the point $p$ a \underline{simple point} of $f$, while if $\alpha(p)<n$, we call $p$ a \underline{multiple point} of $f$.
Note that a multiple point is an isolated and interior point of $\Sigma$

Let $\overrightarrow{{\cal C}^{\prime}}=\overrightarrow{C^{\prime}_{1}}\bigcup \overrightarrow{C^{\prime}_{2}}\bigcup \ldots \bigcup \overrightarrow{C^{\prime}_{s}}$ be a set of oriented and disjoint simple closed curves in a surface $\Sigma$, and let $g$ be a map $g:\Sigma \rightarrow \Sigma$ such that $g(\overrightarrow{{\cal C}^{\prime}})=\overrightarrow{{\cal C}^{\prime}}$ and $g|_{\overrightarrow{{\cal C}}^{\prime}}$ is periodic.
Let $m_{j}$ be the smallest positive integer such that $g^{m_{j}}(\overrightarrow{C^{\prime}_{j}})=\overrightarrow{C^{\prime}_{j}}$.
The restriction $g^{m_{j}}|_{\overrightarrow{C^{\prime}_{j}}}$ is a periodic map of $\overrightarrow{C^{\prime}_{j}}$ of order, say, $\lambda_{j}>0$.
Then $m\lambda_{j}=n$ (= the order of $g|_{\overrightarrow{C_{j}}}$).
Let $q$ be any point on $C^{\prime}_{j}$, and suppose that the images of $q$ under the iteration of $g^{m_{j}}$ are ordered as $(q, g^{m_{j}\sigma_{j}}(q), g^{2m_{j}\sigma_{j}}(q),\ldots, g^{(\lambda_{j}-1)m_{j}\sigma_{j}}(q))$ viewed in the direction of $\overrightarrow{C^{\prime}_{j}}$, where $\sigma_{j}$ is an integer with $0\leq \sigma_{j}\leq \lambda_{j}-1$ and $\gcd(\sigma_{j}, \lambda_{j})=1$ $(j=1,\ldots, s)$.
So $\sigma_{j}=0$ iff $\lambda_{j}=1$.
Let $\delta_{j}$ be the integer which satisfies 
\begin{eqnarray*}
\sigma_{j}\delta_{j}\equiv 1\pmod{\lambda_{j}}, \ \ 0\leq \delta_{j}\leq \lambda_{j}-1 \ \ (j=1,\ldots,s)
\end{eqnarray*}
(NB. $\delta_{j}=0$ iff $\lambda_{j}=1$).
Then the action of $g^{m_{j}}$ on $\overrightarrow{C^{\prime}_{j}}$ is the rotation of angle $2\pi \delta_{j}/\lambda_{j}$ with a suitable parametrization of $\overrightarrow{C^{\prime}_{j}}$ as an oriented circle.
\begin{Def}[{\rm [Ni]}]
The triple $(m_{i},\lambda_{i},\sigma_{i})$ and $(m_{i},\lambda_{i},\delta_{i})$ are called \underline{the valency} and \underline{the second valency} of $\overrightarrow{C^{\prime}_{j}}\in \overrightarrow{{\cal C}^{\prime}}$ \underline{with respect to} $g$.
\end{Def}

Nielsen also defined the \underline{valency of a boundary curve} as its valency with respect to $f$ assuming it has the orientation induced by the surface $\Sigma$.
The \underline{valency of a multiple point} $p$ is defined to be the valency of the boundary curve $\partial D_{p}$, oriented from the outside of a disk neighborhood $D_{p}$ of $p$.

\

Let ${\cal C}=\displaystyle \bigcup_{i=1}^{r}C_{i}$ be an admissible system of cut curves subordinate to a pseudo-periodic map $f:\Sigma_{g+1}\rightarrow \Sigma_{g+1}$.
We assume that $s(C_{i})$ of $f$ is non-zero $(i=1,\ldots,r)$

Thus Matsumoto and Montesinos proved the following proposition :
\begin{prop}[\rm{[MM1]}, Corollary 3.3.1, Corollary 3.7.1]\label{pro1}
Let $A_{i}$ be an annular neighborhood of $C_{i}$ with $\partial A_{i}=\partial_{0}A_{i}\bigcup \partial_{1}A_{i}$.
Let $(m_{i}^{\nu},\lambda_{i}^{\nu},\delta_{i}^{\nu})$ be the second valency of $\partial_{\nu}A_{i}$ with respect to $f$ $(\nu=0,1)$.
The second valency $(m_{i}^{\nu},\lambda_{i}^{\nu},\delta_{i}^{\nu})$ of $\partial_{\nu}A_{i}$ oriented by the orientation induced from $A_{i}$ is defined as the valency of $-\overrightarrow{C_{i}}$ with respect to the $f$.

Then, If $C_{i}$ is non-amphidrome,
\begin{description}
\item[(1)] $m^{0}_{i}=m^{1}_{i}$,
\item[(2)] $s(C_{i})+\delta^{0}_{i}/\lambda^{0}_{i}+\delta^{1}_{i}/\lambda^{1}_{i}$ is an integer.
\end{description}
If $C_{i}$ is amphidrome,
\begin{description}
\item[(1)] $m^{0}_{i}=m^{1}_{i}$= an even number,
\item[(2)] $\delta_{i}^{0}=\delta_{i}^{1}$ and $\lambda_{i}^{0}=\lambda_{i}^{1}$,
\item[(3)] $s(C_{i})/2+\delta_{i}/\lambda_{i}$ is an integer,
\end{description}
$(\lambda_{i}$ denotes $\lambda^{0}_{i}=\lambda^{1}_{i}$, and $\delta_{i}$ denotes $\delta^{0}_{i}=\delta^{1}_{i})$.
\end{prop}

\section{The conjugacy classes of roots of the Dehn twist about a nonseparating curve}
\subsection{Proof of Main Theorem A}
In section 3.1 we do not distinguish a homeomorphism/curve and its isotopy class.

Let $t_{C}$ be the Dehn twist about a simple closed curve $C$ in $\Sigma_{g+1}$.
Suppose that $h$ is a root of $t_{C}$ of degree $n>1$.
Since we have 
\begin{eqnarray*}
t_{C}=h^{n}=hh^{n}h^{-1}=ht_{C}h^{-1}=t_{h(C)},
\end{eqnarray*}
we see that $h(C)=C$.
By $t_{C}=h^{n}$, $h|_{\Sigma_{g+1}-\{C\}}$ must be a periodic map of order $n$.
Therefore $h$ is a pseudo-periodic map, and an admissible system of cut curves is ${\cal C}=C$.

Let $A$ be an annular neighborhood of $C$.
Let $(m^{0},\lambda^{0},\delta^{0})$ and $(m^{1},\lambda^{1},\delta^{1})$ be the second valencies of  $\partial_{0}A$ and $\partial_{1}A$ with respect to $h$.

\begin{claim}\label{cl}
The simple closed curve $C$ is non-amphidrome with respect to $h$.
\end{claim}
\begin{proof}
We assume that $C$ is amphidrome with respect to $h$.

We will determine the screw number $s(C)$ of $h$.
Let $b$ be $\Sigma_{g+1}-C$.
Let $\alpha$, $\beta$ and $n_{b}$ be the smallest positive integers such that $h^{\alpha}(\overrightarrow{C})=\overrightarrow{C}$, $h^{\beta}(b)=b$ and $(h^{\beta}|_{b})^{n_{b}}=id_{b}$.
Since $h(C)=C$ and $C$ is amphidrome, we have $\alpha=m^{0}=m^{1}=2$.
Moreover, we have $\beta=1$, $n_{b}=L$.
Thus, by $h^{n}(\overrightarrow{C})=t_{C}(\overrightarrow{C})=\overrightarrow{C}$, we can denote $n=2k$.
By definition of $L>0$, $L$ is a divisor of $n$ ($z:=n/L \in \mathbb{Z}_{\geq 1}$).
Then, we have $t_{C}=h^{n}=(h^{L})^{z}=(t_{C}^{e})^{z}=t_{C}^{ez}$.
Therefore, we see that $e=1$ and $L=2k=n$.
From the above arguments, we have
\begin{eqnarray*}
s(C)=1/k.
\end{eqnarray*}
By proposition~\ref{pro1}, we have
\begin{eqnarray*}
\delta/\lambda=(2k-1)/2k,
\end{eqnarray*}
(Here $\lambda$ denotes $\lambda^{0}=\lambda^{1}$, and $\delta$ denotes $\delta^{0}=\delta^{1}$.)
However, since $n=2k$ and the action of $h^{2}$ is the rotation of angle $2\pi \delta/\lambda$ in circle, $\delta/\lambda$ must be equal to $\delta/k$.
This is a contradiction.
\end{proof}

\

\begin{lemma}\label{val}
Let $C$ be a nonseparating simple closed curve in $\Sigma_{g+1}$.
Then,
\begin{eqnarray*}\label{rel}
\delta^{0}+\delta^{1}=n-1 \ \ (1\leq \delta^{\nu}\leq n-2 ,\ \nu=0,1)
\end{eqnarray*}
\end{lemma}
\begin{proof}
We use Proposition~\ref{pro1}.
We will determine the screw number $s(C)$ of $h$.
Since $h(C)=C$ and $C$ is non-amphidrome (by Claim~\ref{cl}), we have $\alpha=m^{0}=m^{1}=1$.
Thus, we can find $s(C)=1/n$ by the similar argument of Claim~\ref{cl}.
Furthermore, we can find that $\lambda^{\nu}$=order of $h|_{\overrightarrow{\partial_{\nu}A}}$ (by $m^{\nu}\lambda^{\nu}$=order of $h|_{\overrightarrow{\partial_{\nu}A}}$ for $\nu=0,1$).
By Proposition~\ref{pro1}, we have $\delta^{0}/\lambda^{0} + \delta^{1}/\lambda^{1}=(n-1)/n$.

Let $\partial_{\nu}(\Sigma_{g+1}-A)$ be the boundary components which corresponds to $\partial_{\nu}A$ $(\nu=0,1)$.
Then, $h|_{(\Sigma_{g+1}-A)}$ is periodic map of order $n (= L)$.
Therefore, we can get periodic map $\tilde{h}:\Sigma_{g}\rightarrow \Sigma_{g}$ of period $n$ by pasting two disks $D_{0}$, $D_{1}$ to $\partial_{0}(\Sigma_{g+1}-A)$, $\partial_{1}(\Sigma_{g+1}-A)$.
Since $A$ is non-amphidrome, we can see that $\tilde{h}$ fixes the center points $D_{0}$, $D_{1}$.
Consequently, the period $\lambda^{\nu}$ $(\nu=0,1)$ of $h|_{\partial_{\nu}A}$ is equal to $n$.
We note that $\delta^{\nu}$ is not equal to 0.
Because, if $\delta^{\nu}=0$ then $n=\lambda^{\nu}$ is equal to $1$ by the definition of $\lambda^{\nu}$.
This is in contradiction to $n>1$.
\end{proof}
\begin{remark}
If $C$ is separating, then we may allow that $\lambda^{0}$ is not equal to $\lambda^{1}$, that $\lambda^{0}, \lambda^{1}<n$, and that either $\delta^{0}$ or $\delta^{1}$ is equal to $0$.
In other words, $(1,\lambda^{0},\delta^{0})$, $(1,\lambda^{1},\delta^{1})$ satisfy any one of the following conditions :
\begin{description}
\item[(1)] \ $\delta^{0}/\lambda^{0} + \delta^{1}/\lambda^{1} = (n-1)/n$\ \ $(\delta^{0}, \delta^{1}\neq 0)$ and $n={\rm lcm}(\lambda^{0},\lambda^{1})$,
\item[(2)] \ $\delta^{0}=0$, $\delta^{1}=n-1$, $\lambda^{0}=1$ and $\lambda^{1}=n$,
\item[(3)] \ $\delta^{0}=n-1$, $\delta^{1}=0$, $\lambda^{0}=n$ and $\lambda^{1}=1$.
\end{description}
\end{remark}

Hereafter, $C$ will be a nonseparating curve on $\Sigma_{g+1}$.
Let $\Sigma_{g,2}$ be a compact oriented surface of genus $g$ with two boundary components $\partial_{0}\Sigma_{g,2}$, $\partial_{1}\Sigma_{g,2}$.

\begin{maintheorema}\label{ro}
The conjugacy classes $(\in {\cal M}_{g+1})$ of the roots of $t_{C}$ of degree $n$ correspond to the conjugacy classes $(\in {\cal M}_{g,2})$ of periodic maps of order $n$ which satisfy following condition:

The action on $\overrightarrow{\partial_{\nu}\Sigma_{g,2}}$ $(\nu=0,1)$ is the rotation angle of $2\pi \delta^{\nu}/n$ such that $\delta^{0}+\delta^{1}=n-1$
$(Here \ \gcd(\delta^{\nu},n)=1, 1\leq \delta^{\nu}\leq n-2 \ (\nu=0,1))$.
\end{maintheorema}
\begin{proof}
We will prove Main Theorem A by using Theorem~\ref{MM}.

Let $h$ be a root of Dehn twist $t_{C}$ of degree $n$.
Let $G_{C}$ be the oriented graph $G_{C}$ whose vertices and edges correspond to connected components of $\Sigma_{g+1}-C$ and $C$.
Since $C$ is non-amphidrome, we can find that the action of $h$ on the oriented graph $G_{C}$ is identity.
From the above arguments, we have ${\cal C}=C$, $\alpha=1$, $\beta=1$, $n_{b}=n$, that $C$ is non-amphidrome, and that the action of $h$ on $G_{C}$ is identity.
These are the same data as $h^{-1}$.

Since the screw number $s(C)$ of $h$ is equal to $1/n$, we see that $s(C)$ of $h^{-1}$ is equal to $-1/n$.
Therefore, $h^{-1}$ is negative twist.
If negative twists are restricted to roots $h^{-1}$ of $t^{-1}_{C}$, the conjugacy class of a root $h^{-1}$ of $t^{-1}_{C}$ is determined by $n$ and the conjugacy class of a periodic maps $h^{-1}|_{\Sigma_{g+1}-C}$ by using Theorem~\ref{MM}.
It means that the conjugacy class of a root $h$ of $t_{C}$ is determined by $n$ and the conjugacy class of a periodic map $h|_{\Sigma_{g+1}-C}=h|_{\Sigma_{g,2}}$.
By Lemma~\ref{val} $h|_{\Sigma_{g+1}-C}=h|_{\Sigma_{g,2}}$ satisfies the condition.

In Section 3.2 we show that a root is constructed from a periodic map on $\Sigma_{g,2}$ which satisfies the condition.
\end{proof}

\subsection{Construction of a root of Dehn twist from a periodic map}
Given a periodic map on $\Sigma_{g,2}$ with the condition of Main Theorem A, we construct a root of Dehn twist as follow:

Let $\tilde{f}$ be a periodic map on $\Sigma_{g,2}$ which satisfies the condition.
Figure~\ref{fig2} shows the periodic map $\tilde{f}$.
\begin{figure}[htbp]
 \begin{center}
  \includegraphics*[width=7cm]{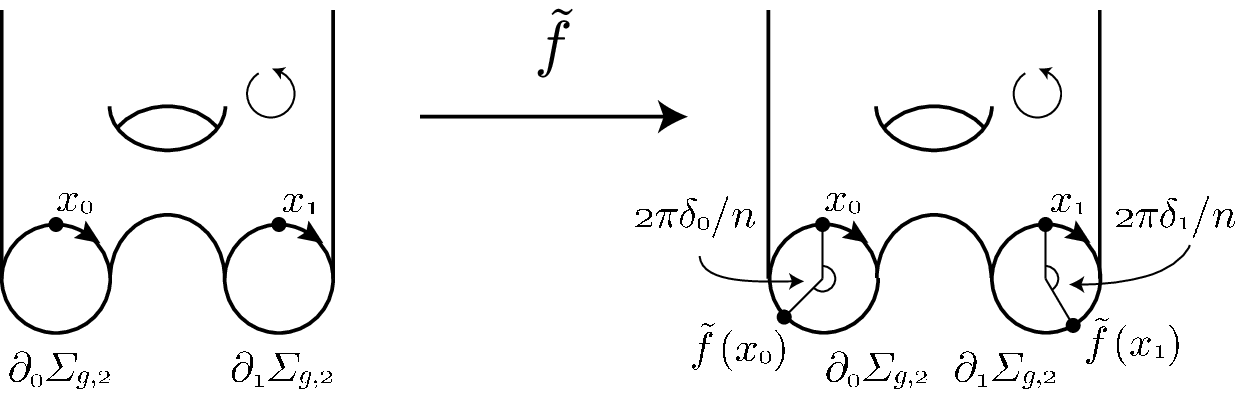}
 \end{center}
 \caption{The periodic map $\tilde{f}$}
 \label{fig2}
\end{figure}

We will extend $\tilde{f}$ to $\Sigma_{g+1}$.
Let us give an orientation to $[0,1]\times S^{1}$ and $\partial_{\nu} \Sigma_{g,2}$ as the figure indicates:
\begin{figure}[htbp]
 \begin{center}
  \includegraphics*[width=1.8cm]{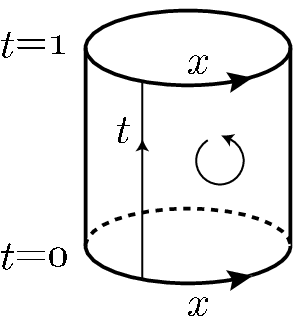}
 \end{center}
 \caption{$[0,1]\times S^{1}$}
 \label{fig3}
\end{figure}

Since $\tilde{f}|_{\partial_{0}\Sigma_{g,2}\bigcup \partial_{1}\Sigma_{g,2}}$ is periodic, there are homeomorphisms $\phi_{0}: \{0\}\times S^{1}\rightarrow \partial_{0} \Sigma_{g,2}$ and $\phi_{1}: \{1\}\times S^{1}\rightarrow \partial_{1} \Sigma_{g,2}$ such that $\tilde{f}\phi_{0}(0,x)=(0,x-\delta^{0}/n)$, $\tilde{f}\phi_{1}(1,x)=(1,x+\delta^{1}/n)$.
Thus, we can get $\Sigma_{g+1}$ by gluing $\{\nu\}\times S^{1}$ to $\partial_{\nu} \Sigma_{g,2}$ by $\phi_{\nu}$ $(\nu=0,1)$.
Let $\phi_{t}$ be a homeomorphism $\phi_{t}:\{t\}\times S^{1}\rightarrow \{t\}\times S^{1}$.
We define $f$ on $\Sigma_{g+1}$ as follow:
\begin{eqnarray*}
f&=&\tilde{f} \ \ on \ \Sigma_{g,2}\\
f\phi_{t}(t, x)&=&\phi_{t}(t, -(\delta^{0}/n)(1-t)x+(\delta^{1}/n)tx) \ \ on \ [0,1]\times S^{1}
\end{eqnarray*}

Figure~\ref{fig4} shows a homeomorphism $f$. 
\begin{figure}[htbp]
 \begin{center}
  \includegraphics*[width=7cm]{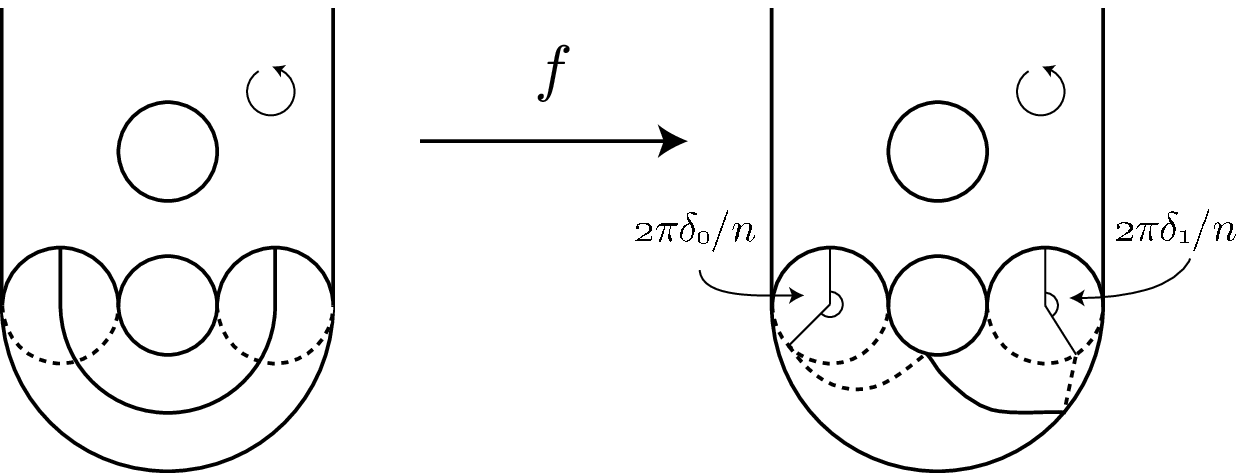}
 \end{center}
 \caption{the homeomorphism $f$}
 \label{fig4}
\end{figure}

Figure~\ref{fig5} shows the homeomorphism $f|_{[0,1]\times S^{1}}$
\begin{figure}[htbp]
 \begin{center}
  \includegraphics*[width=7cm]{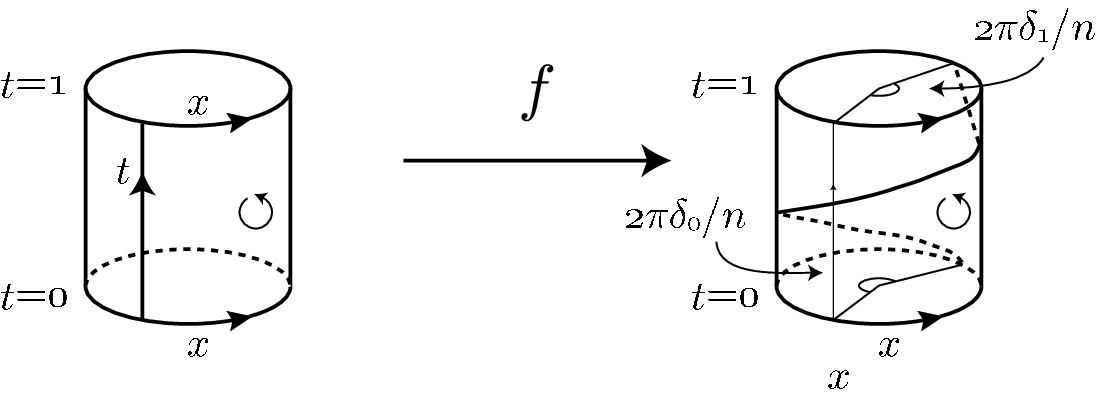}
 \end{center}
 \caption{$f|_{[0,1]\times S^{1}}$}
 \label{fig5}
\end{figure}

Thus, we can see that $f^{n}\phi_{t}(t,x)=\phi_{t}(t,x+(n-1)t-\delta^{0})$.
Therefore, $f^{n}$ is isotopic to $t_{C}^{1-n}$.
So we may change $f$ by isotopy so that $f^{n}=t_{C}^{1-n}$.
By $f([0,1]\times S^{1})=[0,1]\times S^{1}$, $ft_{C}=t_{C}f$.
we have
\begin{eqnarray*}
t_{C}=(t_{C}f)^{n}
\end{eqnarray*}
We have a root $(t_{C}f)$ of $t_{C}$ with degree $n$.
We complete the proof of Main Theorem A.

\subsection{The data set which determines the conjugacy class of a root}

Let $p_{j}$ $(j=1,\ldots,k)$ be a multiple point of a periodic map $\tilde{f}\in {\cal M}_{g,2}$ of Main Theorem A and let $(m_{j},\sigma_{j},\lambda_{j})$ be the valency of $p_{j}$ with respect to $\tilde{f}$.
We note that $m_{j}\lambda_{j}=n$.
Let $(1,\sigma^{\nu},n)$ be the valency of $\partial_{\nu}\Sigma_{g,2}$ $(\nu=0,1)$ respect to $\tilde{f}$.
Note that $\sigma^{\nu}\delta^{\nu}\equiv 1 (\pmod{n})$ and $1\leq \sigma^{\nu}\leq n-1$ (by $n>1$).
By Nielsen's theorem [Ni], it suffices to classify the order of the map, the valencies of multiple points and the valencies of boundary curves.

When we combine the Riemann-Hurwitz formula, Nielsen's theorem and Main Theorem A, we have the following theorem:
\begin{maintheoremb}\label{valency}
The conjugacy class $(\in {\cal M}_{g+1})$ of a root $h$ of $t_{C}$ of degree $n$ is determined by the data $[n,g^{\prime},(\sigma^{0},\sigma^{1});(\sigma_{1},\lambda_{1}),\ldots,(\sigma_{k},\lambda_{k})]$ which satisfies the following conditions:
\begin{description}
\item[(1)] $2g/n=2g^{\prime}+\sum_{j=1}^{k}(1-1/\lambda_{j})$,
\item[(2)] $\sum_{j=1}^{k}\sigma_{j}n/\lambda_{j}+\sigma^{0}+\sigma^{1}\equiv 0(\pmod{n})$,
\item[(3)] $\sigma^{0}+\sigma^{1}+\sigma^{0}\sigma^{1}\equiv 0(\pmod{n})$.
\end{description}
where $g^{\prime}$ is the genus of $\Sigma_{g,2}/\tilde{f}$.
\end{maintheoremb}
McCullough and Rajeevsarathy also got the similar data set in [MR].
The notation $[n,g^{\prime},(\sigma^{0},\sigma^{1});(\sigma_{1},\lambda_{1}),\ldots,(\sigma_{k},\lambda_{k})]$ follows that of [MR].
\begin{corollary}\label{co}
the range of degree $n$ is $3\leq n\leq 2g+1$ and $n$ is odd.
\end{corollary}
\begin{proof}
If $n=2$, by $1\leq \sigma^{0},\sigma^{1}\leq n-1$ we have $\sigma^{0}=\sigma^{1}=1$.
This is in contradiction to the condition (3).
Therefore, we see that $n$ is odd.

For $n=3$, if $g^{\prime}=0$, $k=g$ and $\sigma^{0}=\sigma^{1}=1$, we can select $\sigma_{j}$ $(j=1,\ldots,g)$ which satisfy the condition (2).
It means that there always exists the root of degree 3 for $g\geq 1$.
In next section, we will give Dehn twist expression of the root of degree 3.

We assume that $n>2g+1$.
By the condition (1) we have $1>2g/n=2g^{\prime}+\sum_{j=1}^{k}(1-1/\lambda_{j})$ so $g^{\prime}=0$ and $k=1$.
From the Condition (2) and (3) we have $n\sigma_{1}/\lambda_{1}\equiv \sigma^{0}\sigma^{1} \pmod{n}$.
Therefore, we see that $0\equiv n\sigma_{1}\equiv \sigma^{0}\sigma^{1}\lambda_{1} \pmod{n}$.
It means that $\sigma^{0}\sigma^{1}\lambda_{1}/n=\sigma^{0}\sigma^{1}/m_{1}$ is an integer (we note that $m_{1}\lambda_{1}=n$).
Since $\gcd(\sigma^{0},n)=\gcd(\sigma^{1},n)=1$ and $m_{1}\lambda_{1}=n$, we see that $m_{1}$ must be equal to $1$.
It means that $\lambda_{1}=n$.
Thus we have $2g/n=1-1/n$ so $n=2g+1$.
This is in contradiction to $n>2g+1$.
Since Margalit and Schleimer constructed the root of degree $2g+1$, we have $n\leq 2g+1$.
\end{proof}
McCullough and Rajeevsarathy [MR] also proved by the similar arguments.

\section{Dehn twist expression of the root of degree 3}
We will give Dehn twist expression of the root of degree 3.
The key points are to use the \underline{star relation} which is given by Gervais [Ge].

Consider the torus with 3 boundary components $d_{1}$, $d_{2}$, $d_{3}$ and let $a_{1}$, $a_{2}$, $a_{3}$ and $b$ be simple closed curves in Figure~\ref{star}
\begin{figure}[htbp]
 \begin{center}
  \includegraphics*[width=3cm]{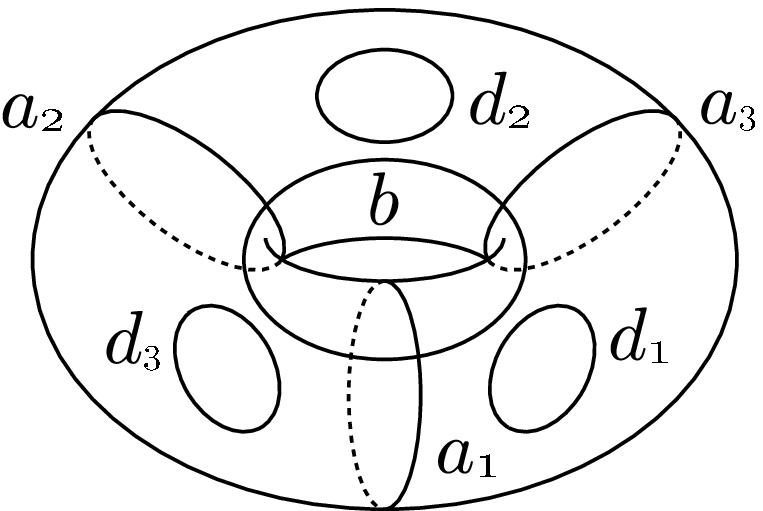}
 \end{center}
 \caption{The curves of star relation}
 \label{star}
\end{figure}

The \underline{star relation} is as follow:
\begin{eqnarray*}
(t_{a_{1}}t_{a_{2}}t_{a_{3}}b)^{3}=t_{d_{1}}t_{d_{2}}t_{d_{3}}
\end{eqnarray*}
If $a_{1}=a_{2}$, then $t_{d_{3}}$ is trivial, and the relation becomes
\begin{eqnarray*}
(t^{2}_{a_{1}}t_{a_{3}}b)^{3}=t_{d_{1}}t_{d_{2}}
\end{eqnarray*}

Let $\alpha_{i}$, $\alpha^{\prime}_{i}$, $\beta_{i}$ $(i=1,\ldots,g+1)$ and $\gamma$ be nonseparating simple closed curves and let $s_{j}$ $(j=1,\ldots,g-1)$ be separating simple closed curves in Figure~\ref{curves}.
\begin{figure}[htbp]
 \begin{center}
  \includegraphics*[width=10cm]{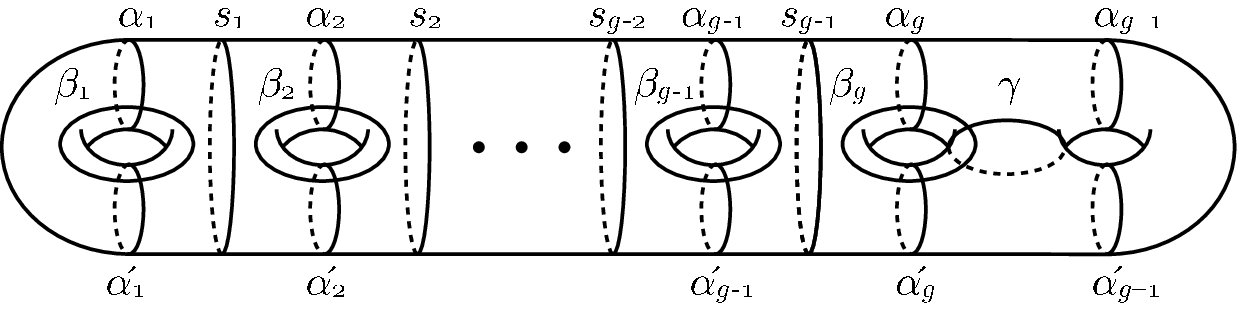}
 \end{center}
 \caption{The curves $\alpha_{i}$, $\alpha^{\prime}$, $\beta_{i}$ $(i=1,\ldots,g+1)$, $\gamma$, $s_{j}$ $(j=1,\ldots,g-1)$}
 \label{curves}
\end{figure}

We define
\begin{eqnarray*}
\rho_{1}&=&(t_{\alpha_{1}}t_{\beta_{1}})^{2}\\
\rho_{i}&=&t_{\alpha_{i}}^{2}t_{\alpha^{\prime}_{i}}t_{\beta_{i}} \ \ (i=2,\ldots,g-1)\\
\rho_{g}&=&t_{\alpha_{g}}t_{\gamma}t_{\alpha^{\prime}_{g}}t_{\beta_{g}}
\end{eqnarray*}
and 
\begin{eqnarray*}
\hat{h}&=&\rho_{g}\rho_{g-1}^{-1}\rho_{g-2}\cdots \rho_{3}^{-1}\rho_{2}\rho_{1}^{-1} \ \ \ (if \ g+1 \ is \ odd),\\
\hat{h}&=&\rho_{g}\rho_{g-1}^{-1}\rho_{g-2}\cdots \rho_{3}\rho_{2}^{-1}\rho_{1} \ \ \ (if \ g+1 \ is \ even).
\end{eqnarray*}
We note $\rho_{1}^{3}=t_{s_{1}}$.
Then, by star relation we have $\hat{h}^{3}=t_{\alpha_{g+1}}^{2}$.
When we define
\begin{eqnarray*}
h=t_{\alpha_{g+1}}\hat{h}^{-1},
\end{eqnarray*}
$h$ is the root of $t_{\alpha_{g+1}}$ of degree 3.

\section{A root of an image of the Dehn twist about a nonseparating curve to ${\rm Sp}(2(g+1),\mathbb{Z})$}
The action of ${\cal M}_{g+1}$ on ${\rm H}_{1}(\Sigma_{g+1};\mathbb{Z})$ preserves the algebraic intersection forms, so it induces a representation $\phi:{\cal M}_{g+1}\rightarrow {\rm Sp}(2(g+1),\mathbb{Z})$, which is well-known to be surjective.
In this section we will prove that if $C$ is nonseparating curve, then an image of Dehn twist $t_{C}$ has no roots of even degree.
We prove the case of $g+1=2$.
An element $2\times 2$ matrix $A\in{\rm Sp}(4,\mathbb{Z})$ satisfies that $AJ{}^{t}\!{A}=J$, where ${}^{t}\!{A}$ is transpose of $A$ and $J$ is
\begin{eqnarray*}
J=\left(
\begin{array}{cccc}
0& 0& 1& 0\\
0& 0& 0& 1\\
-1& 0& 0& 0\\
0& -1& 0& 0
\end{array}
\right).
\end{eqnarray*}
Let $\alpha_{1}$ be a nonseparating simple closed curve in Figure~\ref{curves} and let $S$ be 
\begin{eqnarray*}
S=\rho(t_{\alpha_{1}})=\left(
\begin{array}{cccc}
1& 0& 0& 0\\
0& 1& 0& 0\\
1& 0& 1& 0\\
0& 0& 0& 1
\end{array}
\right),
\ \
(S^{-1}=\left(
\begin{array}{cccc}
1& 0& 0& 0\\
0& 1& 0& 0\\
-1& 0& 1& 0\\
0& 0& 0& 1
\end{array}
\right)).
\end{eqnarray*}
We assume that $S=A^{2}$, where $A(\in {\rm Sp}(2(g+1),\mathbb{Z}))$ is
\begin{eqnarray*}
A=\left(
\begin{array}{cccc}
a_{11}& a_{12}& a_{13}& a_{14}\\
a_{21}& a_{22}& a_{23}& a_{24}\\
a_{31}& a_{32}& a_{33}& a_{34}\\
a_{41}& a_{42}& a_{43}& a_{44}
\end{array}
\right).
\end{eqnarray*}
Since $A^{-1}=S^{-1}A$, we have ${}^{t}\!{A}J=JA^{-1}=JS^{-1}A$.
By
\begin{eqnarray*}
{}^{t}\!{A}J=\left(
\begin{array}{cccc}
-a_{31}& -a_{41}& a_{11}& a_{21}\\
-a_{32}& -a_{42}& a_{12}& a_{22}\\
-a_{33}& -a_{43}& a_{13}& a_{23}\\
-a_{34}& -a_{44}& a_{14}& a_{24}
\end{array}
\right)
\end{eqnarray*}
and
\begin{eqnarray*}
JS^{-1}A=\left(
\begin{array}{cccc}
-a_{11}+a_{31}& -a_{12}+a_{32}& -a_{13}+a_{33}& -a_{14}+a_{34}\\
a_{41}& a_{42}& a_{43}& a_{44}\\
-a_{11}& -a_{12}& -a_{13}& -a_{14}\\
-a_{21}& -a_{22}& -a_{23}& -a_{24}
\end{array}
\right)
\end{eqnarray*}
, we have
\begin{eqnarray}\label{m1}
a_{11}/2=a_{31}.
\end{eqnarray}
Since $SA=AS$, we have
\begin{eqnarray*}
\left(
\begin{array}{cccc}
a_{11}& a_{12}& a_{13}& a_{14}\\
a_{21}& a_{22}& a_{23}& a_{24}\\
a_{11}+a_{31}& a_{12}+a_{32}& a_{13}+a_{33}& a_{14}+a_{34}\\
a_{41}& a_{42}& a_{43}& a_{44}
\end{array}
\right)
=\left(
\begin{array}{cccc}
a_{11}+a_{13}& a_{12}& a_{13}& a_{14}\\
a_{21}+a_{23}& a_{22}& a_{23}& a_{24}\\
a_{31}+a_{33}& a_{32}& a_{33}& a_{34}\\
a_{41}+a_{43}& a_{42}& a_{43}& a_{44}
\end{array}
\right)
\end{eqnarray*}
We have $a_{12}=a_{13}=a_{14}=0$.
By $S=A^{2}$, we have 
\begin{eqnarray}\label{m2}
a_{11}=\pm 1.
\end{eqnarray}
By equations (\ref{m1}) and (\ref{m2}) we have that $a_{31}=a_{11}/2=\pm 1/2$.
This is in contradiction to $A\in {\rm Sp}(4,\mathbb{Z})$.
For $g>2$, we can prove by the same arguments.

Since $S=\rho(t_{C})$ has no roots of degree 2, we can find that a root of $t_{C}$ cannot have even degree.

\section{Remarks}
We will consider a root of the Dehn twist about a nonseparating curve in the mapping class group of a surface with boundary components and punctures.

Let $\Sigma_{g,b_{1},p}^{b_{2}}$ be a compact oriented surface with $b_{1}+b_{2}$ boundary components $\partial_{1},\ldots,\partial_{b_{1}}$, $\partial^{\prime}_{1},\ldots,\partial^{\prime}_{b_{2}}$ and $p$ punctures.
We denote ${\cal M}_{g,b_{1},p}^{b_{2}}$ the mapping class group of $\Sigma_{g,b_{1},p}^{b_{2}}$, the group of isotopy classes of orientation-preserving homeomorphisms which are allowed to permute $p$ punctures and the points on $\partial_{1},\ldots,\partial_{b_{1}}$, and restrict to the identity on $\partial^{\prime}_{1},\ldots,\partial^{\prime}_{b_{2}}$.
The isotopies are also required to permute $p$ punctures and the points on $\partial_{1},\ldots,\partial_{b_{1}}$ and to fix the points on $\partial^{\prime}_{1},\ldots,\partial^{\prime}_{b_{2}}$.
Let $t_{C}$ be the Dehn twist about a nonseparating curve $C$ in $\Sigma_{g+1,b_{1},p}^{b_{2}}$ $(g+1\geq 2)$.
In this section we do not distinguish a homeomorphism/curve and its isotopy class.

We assume that $t_{C}=h^{n}\in {\cal M}_{g+1,b_{1},p}^{b_{2}}$.
Since $id|_{\Sigma_{g+1,b_{1},p}^{b_{2}}-C}=t_{C}|_{\Sigma_{g+1,b_{1},p}^{b_{2}}-C}=h^{n}|_{\Sigma_{g+1,b_{1},p}^{b_{2}}-C}$, we see that $h|_{\Sigma_{g+1,b_{1},p}^{b_{2}}-C}$ is a periodic map in ${\cal M}_{g,b_{1}+2,p}^{b_{2}}$.
However, if $b_{2}>0$ then ${\cal M}_{g,b_{1}+2,p}^{b_{2}}$ is torsion free.
This is in contradiction to what $h|_{\Sigma_{g+1,b_{1},p}^{b_{2}}-C}$ is a periodic map in ${\cal M}_{g,b_{1}+2,p}^{b_{2}}$.
Therefore, we see that if $b_{2}>0$, $t_{C}$ has no roots in ${\cal M}_{g+1,b_{1},p}^{b_{2}}$.

We assume that $b_{2}=0$ and omit $b_{2}$ from the notation.

Let $h(\in {\cal M}_{g+1})$ be a root of $t_{C}(\in {\cal M}_{g+1})$ of degree $2g+1$.
By Main Theorem A and the proof of Corollary C, we see that $h|_{\Sigma_{g+1}-C}=h|_{\Sigma_{g,2,0}}$ is a periodic map of order $2g+1$ such that $h|_{\Sigma_{g,2,0}}$ maps $\partial_{\nu}\Sigma_{g,2,0}$ to $\partial_{\nu}\Sigma_{g,2,0}$ $(\nu=0,1)$ and has only one fixed point.
Therefore, if $b_{1}+p\equiv r \pmod{2g+1}$ and $1<r<2g+1$ then there is no $\mathbb{Z}_{2g+1}$-action on $\Sigma_{g,b_{1}+2,p}$.
It means that $t_{C}\in {\cal M}_{g+1,b_{1},p}$ has no roots of degree $2g+1$.
Therefore, we see that if $b_{1}+p\equiv r\pmod{2g+1}$ and $1<r<2g+1$, $t_{C}\in {\cal M}_{g+1,b_{1},p}$ has no roots of degree $2g+1$.
In particular, if $g=1$ and $b_{1}+p\equiv 2\pmod{3}$ then $t_{C}\in {\cal M}_{2,b_{1},p}$ has no roots.

\ \\
Department of Mathematics, Graduate School of Science, Osaka University, Toyonaka, Osaka 560-0043, Japan\\
\textit{E-mail address:} \bf{n-monden@cr.math.sci.osaka-u.ac.jp}
\end{document}